\documentclass[12pt]{amsart}  
\usepackage[latin1]{inputenc}
\usepackage{amsmath} 
\usepackage{amsfonts}
\usepackage{amssymb}
\usepackage{latexsym} 
\usepackage{graphicx}
\usepackage{subfigure}
\usepackage{color}
\usepackage{hyperref}
\usepackage{verbatim}
\usepackage[all]{xy}
\usepackage{graphics}
\usepackage{pdfsync}
\usepackage{youngtab}
\usepackage{ytableau}
\usepackage{cancel} 
\theoremstyle{plain}
\newtheorem{theorem}{Theorem}[section]
\newtheorem{proposition}[theorem]{Proposition}

\newtheorem{lemma}[theorem]{Lemma}
\newtheorem{corollary}[theorem]{Corollary}

\newtheorem{definition}[theorem]{Definition}
\newtheorem{example}[theorem]{Example}

\theoremstyle{remark}
\newtheorem{remark}[theorem]{Remark}

\numberwithin{equation}{section}
\newcommand{\bC}{\mathbb{C}}

\newcommand{\coverc}{\lessdot _{c}} 

\newcommand{\lessc}{< _{c}} 
\newcommand{\Lc}{\mathcal{L}_{c}} 

\newcommand{\Qsym}{\ensuremath{\operatorname{QSym}}}
\newcommand{\qs}{{\mathcal{S}}}		

\newcommand{\Nsym}{\ensuremath{\operatorname{NSym}}}
\newcommand{\ncs}{{\mathbf{s}}}          	

\newcommand{\ncsa}{\mathbf{s}_{\alpha}}
\newcommand{\ncsb}{\mathbf{s}_{\beta}}

\newcommand{\set}{\mathrm{set}} 
\newcommand{\des}{\mathrm{Des}} 
\newcommand{\comp}{\mathrm{comp}} 
\newcommand{\rem}{\mathfrak{d}} 
\newcommand{\addt}{\mathfrak{t}} 
\newcommand{\addu}{\mathfrak{u}} 

\newcommand{\Id}{{Id}} 
\newcommand{\down}{\mathfrak{d}}
\newcommand{\downrow}{\mathfrak{d}}

\newcommand{\alphap}{\alpha^+}
\newcommand{\betam}{\beta^-}

\newcommand{\rtau}{{\tau}} 

\newcommand{\suchthat}{\;|\;}

\newcommand{\cskew}{{/\!\!/}}

\newcommand{\spam}{\operatorname{span}}

\newlength\cellsize 
\savebox2{%
\begin{picture}(15,15)
\put(0,0){\line(1,0){15}}
\put(0,0){\line(0,1){15}}
\put(15,0){\line(0,1){15}}
\put(0,15){\line(1,0){15}}
\end{picture}}
\newcommand\cellify[1]{\def\thearg{#1}\def\nothing{}%
\ifx\thearg\nothing
\vrule width0pt height\cellsize depth0pt\else
\hbox to 0pt{\usebox2\hss}\fi%
\vbox to 15\unitlength{
\vss
\hbox to 15\unitlength{\hss$#1$\hss}
\vss}}
\newcommand\tableau[1]{\vtop{\let\\=\cr
\setlength\baselineskip{-16000pt}
\setlength\lineskiplimit{16000pt}
\setlength\lineskip{0pt}
\halign{&\cellify{##}\cr#1\crcr}}}
\savebox3{%
\begin{picture}(15,15)
\put(0,0){\line(1,0){15}}
\put(0,0){\line(0,1){15}}
\put(15,0){\line(0,1){15}}
\put(0,15){\line(1,0){15}}
\end{picture}}
\newcommand\expath[1]{%
\hbox to 0pt{\usebox3\hss}%
\vbox to 15\unitlength{
\vss
\hbox to 15\unitlength{\hss$#1$\hss}
\vss}}
\newcommand\bas[1]{\omit \vbox to \cellsize{ \vss \hbox to \cellsize{\hss$#1$\hss} \vss}}

\begin{document}

\title[Generalized skew Pieri rules]{Quasisymmetric and noncommutative skew Pieri rules}

\author{V. Tewari}
\address{Department of Mathematics, University of Washington, Seattle, WA 98105, USA}
\email{\href{mailto:vasut@math.washington.edu}{vasut@math.washington.edu}}
\author{S. van Willigenburg}
\address{Department of Mathematics, University of British Columbia, Vancouver, BC V6T 1Z2, Canada}
\email{\href{mailto:steph@math.ubc.ca}{steph@math.ubc.ca}}

\thanks{
The authors were supported in part by the National Sciences and Engineering Research Council of Canada.}
\subjclass[2010]{Primary 05E05, 16T05, 16W55; Secondary 05A05,  05E10}
\keywords{composition, composition poset, composition tableau,   noncommutative symmetric function, quasisymmetric function,  skew Pieri rule}

\begin{abstract}
In this note we derive skew Pieri rules in the spirit of Assaf-McNamara for skew quasisymmetric Schur functions using the Hopf algebraic techniques of Lam-Lauve-Sottile, and recover the original rules of Assaf-McNamara as a special case. We then apply these techniques a second time to obtain skew Pieri rules for skew noncommutative Schur functions.
\end{abstract}

\maketitle

\section{Introduction}\label{sec:intro} 
The Hopf algebra of quasisymmetric functions, $\Qsym$, was first defined explicitly in \cite{gessel}. It is a nonsymmetric generalization of the Hopf algebra of symmetric functions, and arises in many areas such as the representation theory of the 0-Hecke algebra \cite{BBSSZ, DKLT, konig, 0-Hecke}, probability \cite{hersh-hsiao, stanley-riffle},  and is the terminal object in the category of combinatorial Hopf algebras \cite{aguiar-bergeron-sottile}. Recently a basis of $\Qsym$, known as the basis of quasisymmetric Schur functions, was discovered \cite{QS}, which is a nonsymmetric generalization of the symmetric function basis of Schur functions. These quasisymmetric Schur functions arose from the combinatorics of Macdonald polynomials \cite{HHL}, have been used to resolve the conjecture that $\Qsym$ over the symmetric functions has a stable basis \cite{lauve-mason}, and  have initiated the dynamic research area of discovering other quasisymmetric Schur-like bases such as row-strict quasisymmetric Schur functions \cite{ferreira, mason-remmel},  Young quasisymmetric Schur functions \cite{LMvW}, dual immaculate  quasisymmetric functions \cite{BBSSZ}, type $B$ quasisymmetric Schur functions \cite{jingli, oguz}, quasi-key polynomials \cite{assafsearles, searles} and quasisymmetric Grothendieck polynomials \cite{monical}. Their name was aptly chosen since these functions not only naturally refine Schur functions, but also generalize many classical Schur function properties, such as the Littlewood-Richardson rule from the classical \cite{littlewood-richardson} to the generalized \cite[Theorem 3.5]{BLvW}, the  Pieri rules from the classical \cite{pieri} to the generalized \cite[Theorem 6.3]{QS} and the RSK algorithm from the classical \cite{knuth, robinson, schensted} to the generalized \cite[Procedure 3.3]{mason}.

Dual to $\Qsym$ is the Hopf algebra of noncommutative symmetric functions, $\Nsym$ \cite{GKLLRT}, whose basis dual to that of quasisymmetric Schur functions is the basis of noncommutative Schur functions \cite{BLvW}. By duality this basis again has a Littlewood-Richardson rule and RSK algorithm, and, due to noncommutativity, two sets of Pieri rules, one arising from multiplication on the right \cite[Theorem 9.3]{tewari} and one arising from multiplication on the left \cite[Corollary 3.8]{BLvW}. Therefore in both $\Qsym$ and $\Nsym$ a key question in this realm remains: Are there \emph{skew} Pieri rules for quasisymmetric and noncommutative Schur functions? In this note we give such rules that are analogous to that of their namesake Schur functions.

More precisely, the note is structured as follows. In Section~\ref{sec:comps} we review necessary notions on compositions and define operators on them. In Section~\ref{sec:QSYMNSYM} we recall $\Qsym$ and $\Nsym$, the bases of quasisymmetric Schur functions and noncommutative Schur functions, and their respective Pieri rules. In Section~\ref{sec:skew} we give skew Pieri rules for quasisymmetric Schur functions in Theorem~\ref{the:QSskewPieri} and recover the Pieri rules for skew shapes of Assaf and McNamara in Corollary~\ref{cor:AM}. We close with skew Pieri rules for noncommutative Schur functions in Theorem~\ref{the:NCskewPieri}.


\section{Compositions and diagrams}\label{sec:comps} A finite list of integers $\alpha = (\alpha _1, \ldots , \alpha _\ell)$ is called a \emph{weak composition} if $\alpha _1, \ldots , \alpha _\ell$ are nonnegative, is called a \emph{composition} if $\alpha _1, \ldots , \alpha _\ell$ are positive, and is called a \emph{partition} if $\alpha _1\geq \cdots \geq\alpha _\ell >0$. Note that every weak composition has an underlying composition, obtained by removing every zero, and in turn every composition has an underlying partition, obtained by reordering the list of integers into weakly decreasing order. Given $\alpha = (\alpha _1, \ldots , \alpha _\ell)$ we call the $\alpha _i$ the \emph{parts} of $\alpha$, also $\ell$ the \emph{length} of $\alpha$ denoted by $\ell(\alpha)$, and the sum of the parts of $\alpha$ the \emph{size} of $\alpha$ denoted by $|\alpha |$. The empty composition of length and size zero is denoted by $\emptyset$. If there exists $\alpha _{k+1} = \cdots = \alpha _{k+j} = i$ then we often abbreviate this to $i^j$. Also, given  weak compositions $\alpha= (\alpha_1,\ldots ,\alpha_{\ell})$ and $\beta=(\beta_1,\ldots,\beta_m)$, we define the \emph{concatenation} of $\alpha$ and $\beta$, denoted by $\alpha \beta$, to be the weak composition $(\alpha_1,\ldots,\alpha_\ell,\beta_1,\ldots,\beta_m)$. We define the \emph{near-concatenation} of $\alpha$ and $\beta$, denoted by $\alpha \odot \beta$, to be the weak composition $(\alpha_1,\ldots,\alpha_\ell + \beta_1,\ldots,\beta_m)$.
For example, if  $\alpha=(2,1,0,3)$ and $\beta=(1,4,1)$, then $\alpha\beta=(2,1,0,3,1,4,1)$ and $\alpha\odot \beta=(2,1,0,4,4,1)$.

The \emph{composition diagram} of a weak composition $\alpha$, also denoted by $\alpha$, is the array of left-justified boxes with $\alpha _i$ boxes in row $i$ from the \emph{top}, that is, following English notation for Young diagrams of partitions. We will often think of $\alpha$ as both a weak composition and as a composition diagram simultaneously, and hence future computations such as adding/subtracting 1 from the rightmost/leftmost part equalling $i$ (as a weak composition) are synonymous with adding/removing a box from the bottommost/topmost row of length $i$ (as a composition diagram).

\begin{example}\label{ex:comps}
The composition diagram of the weak composition of length 5, $\alpha=(2,0,4,3,6)$, is shown below.

$$\tableau{\ &\ \\ \\ \ &\ &\ &\ \\ \ &\ &\ \\\ &\ &\ &\ &\ &\ }$$

The  composition of length 4 underlying $\alpha$ is $(2,4,3,6)$, and the partition of length 4 underlying it is $(6,4,3,2)$. They all have size 15.
\end{example}


\subsection{Operators on compositions}\label{sec:ops} In this subsection we will recall four families of operators, each of which are dependent on a positive integer parameter. These families have already contributed to the theory of quasisymmetric and noncommutative Schur functions, and will continue to cement their central role as we shall see later. Although originally defined on compositions, we will define them in the natural way on weak compositions to facilitate easier proofs. The first of these operators is the box removing operator $\down$, which first appeared in the Pieri rules for quasisymmetric Schur functions \cite{QS}. The second of these is the appending operator $a$. These combine to define our third operator, the jeu de taquin or jdt operator $\addu$. This operator is pivotal in describing jeu de taquin slides on tableaux known as semistandard reverse composition tableaux and in describing the right Pieri rules for noncommutative Schur functions \cite{tewari}. Our fourth and final operator is the box adding operator $\addt$ \cite{BLvW, MNtewari}, which plays the same role in the left Pieri rules for noncommutative Schur functions \cite{BLvW} as $\addu$ does in the aforementioned right Pieri rules. Each of these operators is defined on weak compositions for every  integer $i\geq 0$. We note that
$$\down _0 = a_0 = \addu _0 = \addt _ 0 = \Id$$namely the identity map, which fixes the weak composition it is acting on. With this in mind we now define the remaining operators for $i\geq 1$.

The first \emph{box removing operator} on weak compositions, $\down _i$ for $i\geq 1$, is defined as follows. Let $\alpha$ be a weak composition. Then
$$\down _i (\alpha) = \alpha '$$where $\alpha '$ is the weak composition obtained by subtracting 1 from the rightmost part equalling $i$ in $\alpha$. If there is no such part then we define $\down _i(\alpha) = 0$.

\begin{example}\label{ex:down}
Let $\alpha=(2,1,2)$. Then $\down_1(\alpha)=(2,0,2)$ and $\down_2(\alpha)=(2,1,1)$.
\end{example}

Now we will discuss two notions that will help us state our theorems in a concise way later, as well as connect our results to those in the classical theory of symmetric functions. Let $i_1 < \cdots < i_k$ be a sequence of positive integers, and let $\alpha$ be a weak composition. Consider the operator $\down_{i_1}\cdots \down_{i_k}$ acting on the weak composition $\alpha$, and assume that the result is a valid weak composition. Then the boxes that are removed from $\alpha$ are said to form a \emph{$k$-horizontal strip}, and we can think of the operator $\down_{i_1}\cdots \down_{i_k}$ as removing a $k$-horizontal strip. Similarly, given a sequence of positive integers $i_1\geq \cdots \geq i_k$, consider the operator $\down_{i_1}\cdots \downrow_{i_k}$ acting on $\alpha$ and suppose that the result is a valid weak composition. Then the boxes that are removed from $\alpha$ are said to form a \emph{$k$-vertical strip}. As before, we can think of the operator $\down_{i_1}\cdots \downrow_{i_k}$ as removing a $k$-vertical strip.

\begin{example}\label{ex:horizontal and vertical strip}
Consider $\alpha=(2,5,1,3,1)$.  When we compute $\down_{1}\down_2\down_4\down_5(\alpha)$, the operator $\down_{1}\down_2\down_4\down_5$ removes the $4$-horizontal strip shaded in red from $\alpha$.

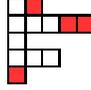
$$
\ytableausetup{smalltableaux,boxsize=0.5em}
\begin{ytableau}
*(white) & *(red!80)\\
*(white) &*(white) &*(white) &*(red!80) &*(red!80) \\
*(white) \\
*(white) &*(white) &*(white) \\
*(red!80)\\
\end{ytableau}
$$

When we compute $\down_3\down_2\down_1\down_1(\alpha)$, the operator $\down_3\down_2\down_1\down_1$ removes the $4$-vertical strip shaded in red from $\alpha$.

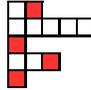
$$
\ytableausetup{smalltableaux,boxsize=0.5em}
\begin{ytableau}
*(white) & *(red!80)\\
*(white) &*(white) &*(white) &*(white) &*(white) \\
*(red!80) \\
*(white) &*(white) &*(red!80) \\
*(red!80)\\
\end{ytableau}
$$

\end{example}

\begin{remark}\label{rem:horizontal and vertical strip}
If we consider partitions as Young diagrams in English notation, then the above notions of horizontal and vertical strips  coincide with their classical counterparts. For example, consider the operator $\down_1\down_2\down_4\down_5$ acting on the partition $(5,3,2,1,1)$, in contrast to acting on the composition $(2,5,1,3,1)$ as in Example \ref{ex:horizontal and vertical strip}. Then the  $4$-horizontal strip shaded in red is removed.

$$
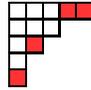
\ytableausetup{smalltableaux,boxsize=0.5em}
\begin{ytableau}
*(white) &*(white) &*(white) &*(red!80) &*(red!80) \\
*(white) &*(white) &*(white) \\
*(white) & *(red!80)\\
*(white) \\
*(red!80)
\end{ytableau}
$$

\end{remark}

We now define the second \emph{appending operator} on weak compositions, $a_i$ for $i\geq 1$, as follows. Let $\alpha = (\alpha _1, \ldots , \alpha _{\ell(\alpha)})$ be a weak composition. Then
$$a _i (\alpha) = (\alpha _1, \ldots , \alpha _{\ell(\alpha)}, i)$$namely, the weak composition obtained by appending a part $i$ to the end of $\alpha$. 

\begin{example}\label{ex:append} Let $\alpha = (2,1,3)$. Then $a_2 ((2,1,3))= (2,1,3,2)$. Meanwhile, $a_j \down _2 ((3,5,1)) = 0$ for all $j\geq 0$ since $\down _2 ((3,5,1)) = 0$.
\end{example}

With the definitions of $a_i$ and $\down _i$ we define the third \emph{jeu de taquin} or \emph{jdt operator} on weak compositions, $\addu _i$ for $i\geq 1$, as
$$\addu _i = a_i \down_1\down_2\down_3 \cdots \down _{i-1}.$$
 
 \begin{example}\label{ex:jdt}
We will compute $\addu_4(\alpha)$ where $\alpha = (3,5,2,4,1,2)$. This corresponds to computing $a_4\down_1\down_2\down_3(\alpha)$. Now 
\begin{eqnarray*}
\down_1\down_2\down_3(\alpha)&=&\down_1\down_2\down_3((3,5,2,4,1,2))\\&=&\down_1\down_2((2,5,2,4,1,2))\\&=& \down_1((2,5,2,4,1,1))\\&=& (2,5,2,4,1,0).
\end{eqnarray*} Hence $\addu_4(\alpha)=(2,5,2,4,1,0,4)$.
\end{example}

Let $i_1 < \cdots < i_k$ be a sequence of positive integers, and let $\alpha$ be a weak composition. Consider the operator $\addu_{i_k}\cdots \addu_{i_1}$ acting on the weak composition $\alpha$, and assume that the result is a valid weak composition. Then the boxes that are added to $\alpha$ are said to form a \emph{$k$-right horizontal strip}, and we can think of the operator $\addu_{i_k}\cdots \addu_{i_1}$ as adding a $k$-right horizontal strip. Similarly, given a sequence of positive integers $i_1\geq \cdots \geq i_k$, consider the operator $\addu_{i_k}\cdots \addu_{i_1}$ acting on $\alpha$ and suppose that the result is a valid weak composition. Then the boxes that are added to $\alpha$ are said to form a \emph{$k$-right vertical strip}. As before, we can think of the operator $\addu_{i_k}\cdots \addu_{i_1}$ as adding a $k$-right vertical strip.

Lastly, we define the fourth \emph{box adding operator} on weak compositions, $\addt _i$ for $i\geq 1$, as follows. Let $\alpha = (\alpha _1, \ldots , \alpha _{\ell(\alpha)})$ be a weak composition. Then
$$\addt _1 (\alpha) =  (1, \alpha _1, \ldots , \alpha _{\ell(\alpha)})$$and for $i\geq 2$
$$\addt _i (\alpha) =  (\alpha _1, \ldots , \alpha _j + 1, \ldots ,\alpha _{\ell(\alpha)})$$where $\alpha _j$ is the leftmost part equalling $i-1$ in $\alpha$. If there is no such part, then we define $\addt _i (\alpha) = 0$.

\begin{example}\label{ex:boxadd}
Consider the composition $\alpha=(3,2,3,1,2)$. Then $\addt_1(\alpha)=(1,3,2,3,1,2)$, $\addt_2(\alpha)=(3,2,3,2,2)$, $\addt_3(\alpha)=(3,3,3,1,2)$, $\addt_4(\alpha)=(4,2,3,1,2)$ and $\addt_i(\alpha)=0$ for all $i\geq 5$.
\end{example}

As with the jdt operators let $i_1 < \cdots < i_k$ be a sequence of positive integers, and let $\alpha$ be a weak composition. Consider the operator $\addt_{i_k}\cdots \addt_{i_1}$ acting on the weak composition $\alpha$, and assume that the result is a valid weak composition. Then the boxes that are added to $\alpha$ are said to form a \emph{$k$-left horizontal strip}, and we can think of the operator $\addt_{i_k}\cdots \addt_{i_1}$ as adding a $k$-left horizontal strip. Likewise, given a sequence of positive integers $i_1\geq \cdots \geq i_k$, consider the operator $\addt_{i_k}\cdots \addt_{i_1}$ acting on $\alpha$ and suppose that the result is a valid weak composition. Then the boxes that are added to $\alpha$ are said to form a \emph{$k$-left vertical strip}, and we can think of the operator $\addt_{i_k}\cdots \addt_{i_1}$ as adding a $k$-left vertical strip.

The box adding operator is also needed to define the composition poset \cite[Definition 2.3]{BLvW}, which in turn will be needed to define skew quasisymmetric Schur functions in the next section.

\begin{definition}\label{def:RcLc} The  \emph{composition poset}, denoted by $\Lc$, is the poset consisting of the set of all compositions equipped with cover relation $\coverc$ such that for compositions  $\alpha, \beta$
$$\beta \coverc \alpha \mbox{ if and only if } \alpha = \addt _i (\beta)$$for some $i\geq1$.
\end{definition}

The order relation  $\lessc$ in $\Lc$ is obtained by taking the transitive closure of the cover relation  $\coverc$.

\begin{example}\label{ex:boxaddLc}
We have that $(3,2,3,1,2) \coverc (4,2,3,1,2)$ by Example~\ref{ex:boxadd}.
\end{example}

\section{Quasisymmetric and noncommutative symmetric functions}\label{sec:QSYMNSYM}
We now recall the basics of graded Hopf algebras before focussing on the graded Hopf algebra of quasisymmetric functions \cite{gessel} and its dual, the graded Hopf algbera of noncommutative symmetric functions \cite{GKLLRT}. We say that $\mathcal{H}$ and $\mathcal{H}^*$ form a pair of  dual graded Hopf algebras each over a field $K$ if there exists a duality pairing $\langle \ ,\ \rangle : \mathcal{H}\otimes \mathcal{H}^{*} \longrightarrow K$,  for which the structure of $\mathcal{H}^*$ is dual to $\mathcal{H}$ that respects the grading, and vice versa. More precisely, the duality pairing pairs the elements of any basis $\{B_i\}_{i\in I}$ of the graded piece $\mathcal{H}^N$ for some index set $I$, and the elements of its dual basis $\{D_i\}_{i\in I}$ of the graded piece $(\mathcal{H}^N)^*$,
given by
$\langle B_i, D_j\rangle = \delta_{ij}$, where the \emph{Kronecker delta}\index{Kronecker delta} $\delta_{ij} = 1$ if $i=j$ and 0 otherwise.
Duality is exhibited in that the product coefficients of one basis are the coproduct coefficients of its dual basis and vice versa, that is, \begin{eqnarray*}
B_i  \cdot B_j = \sum_k b^k_{i,j} B_k &\qquad \Longleftrightarrow\qquad &
\Delta D_k = \sum_{i,j} b^k_{i,j} D_i \otimes D_j  \\
D_i   \cdot D_j = \sum_k d^k_{i,j} D_k &\qquad \Longleftrightarrow\qquad &
\Delta B_k = \sum_{i,j} d^k_{i,j} B_i \otimes B_j
\end{eqnarray*}where $\cdot$ denotes \emph{product} and $\Delta$ denotes \emph{coproduct}. Graded Hopf algebras also have an \emph{antipode} $S: \mathcal{H}\longrightarrow\mathcal{H}$, whose general definition we will not need. Instead we will state the specific antipodes, as needed, later.  Lastly, before we define our specific graded Hopf algebras, we recall one Hopf algebraic lemma, which will play a key role later. For $h\in \mathcal{H}$ and $a\in \mathcal{H}^{*}$, let the following be the respective coproducts in Sweedler notation.
\begin{eqnarray}\label{eq:coproductH}
\Delta (h)&=& \displaystyle\sum_{h} h_{(1)}\otimes h_{(2)}
\end{eqnarray}
\begin{eqnarray}\label{eq:coproductHdual}
\Delta (a)&=&\displaystyle\sum_{a} a_{(1)}\otimes a_{(2)}
\end{eqnarray}
Now define left actions of $\mathcal{H}^{*}$ on $\mathcal{H}$ and $\mathcal{H}$ on $\mathcal{H}^{*}$, both denoted by $\rightharpoonup$, as 
\begin{eqnarray}\label{eq:HdualactingonH}
a \rightharpoonup h &=& \displaystyle\sum_{h}\langle h_{(2)},a\rangle h_{(1)}, 
\end{eqnarray}
\begin{eqnarray}\label{eq:HactingonHdual}
h\rightharpoonup a &=& \displaystyle\sum_{a} \langle h,a_{(2)}\rangle a_{(1)},
\end{eqnarray}
where $a\in \mathcal{H}^{*}$, $h\in \mathcal{H}$. Then we have the following.

\begin{lemma}\cite{lam-lauve-sottile}\label{lem:magiclemma}
For all $g,h \in \mathcal{H}$ and $a\in \mathcal{H}^{*}$, we have that
\begin{eqnarray*}
(a\rightharpoonup g)\cdot h &= & \displaystyle\sum_{h} \left( S(h_{(2)})\rightharpoonup a \right)\rightharpoonup \left(g\cdot h_{(1)}\right)
\end{eqnarray*}
where $S:\mathcal{H}\longrightarrow \mathcal{H}$ is the  antipode.
\end{lemma}

The graded Hopf algebra of quasisymmetric functions, $\Qsym$ \cite{gessel}, is a subalgebra of $\bC [[x_1, x_2, \ldots]]$ with a basis given by the following functions, which in turn are reliant on the natural bijection between compositions and sets, for which we first need to recall that $[i]$ for  $i\geq 1$ denotes the set $\{1,2,\ldots , i\}$. Now we can state the bijection. Given a composition $\alpha = ( \alpha _1 , \ldots , \alpha _{\ell(\alpha)})$, there is a natural subset of $[|\alpha|-1]$ corresponding to it, namely, 
$$\set (\alpha) = \{ \alpha _1 , \alpha _1 + \alpha _2, \ldots , \alpha _1+\alpha _2 + \cdots + \alpha _{\ell(\alpha)-1}\} \mbox{ and } \set((|\alpha|))=\emptyset.$$Conversely, given a subset $S = \{ s_1< \cdots < s_{|S|}\}\subseteq [N-1]$, there is a natural composition of size $N$ corresponding to it, namely, $$\comp (S) = (s_1, s_2 - s_1,  \ldots , N-s_{|S|}) \mbox{ and } \comp(\emptyset)=(N).$$

\begin{definition}\label{def:Fbasis}
Let $\alpha = (\alpha _1, \ldots , \alpha _{\ell(\alpha)})$ be a composition. Then the \emph{fundamental quasisymmetric function} $F_\alpha$ is defined to be $F_\emptyset = 1$ and
$$F_\alpha = \sum x_{i_1} \cdots x_{i_{|\alpha|}}$$where the sum is over all $|\alpha|$-tuples $(i_1, \ldots , i_{|\alpha|})$ of indices satisfying
$$i_1\leq \cdots \leq i_{|\alpha|} \mbox{ and } i_j<i_{j+1} \mbox{ if } j \in \set(\alpha).$$
\end{definition}

\begin{example}\label{ex:Fbasis}
$F_{(1,2)} = x_1x_2^2 + x_1x_3^2 + \cdots + x_1x_2x_3 + x_1x_2x_4 + \cdots.$
\end{example}

Then $\Qsym$ is a graded Hopf algebra 
$$\Qsym = \bigoplus _{N\geq 0} \Qsym ^N$$where
$$\Qsym ^N = \spam \{ F_\alpha \suchthat |\alpha| = N \}.$$The product for this basis is inherited from the product of monomials and Definition~\ref{def:Fbasis}. The coproduct \cite{gessel} is given by $\Delta(1)=1\otimes1$ and
\begin{equation}\label{eq:Fcoproduct}
\Delta(F_\alpha)= \sum _{\beta\gamma = \alpha \atop \mbox{ or }\beta\odot\gamma = \alpha}F_\beta \otimes F_\gamma
\end{equation}and the antipode, which was discovered independently in \cite{ehrenborg-1, malvenuto-reutenauer}, is given by $S(1)=1$ and
\begin{equation}\label{eq:antipode}
S(F_\alpha)= (-1)^{|\alpha|}F_{\comp(\set(\alpha)^c)}
\end{equation}where $\set(\alpha)^c$ is the complement of $\set(\alpha)$ in the set $[|\alpha| -1]$.

\begin{example}\label{ex:Fcoprodantipode}
$$\Delta (F_{(1,2)})=F_{(1,2)}\otimes 1 + F_{(1,1)}\otimes F_{(1)} + F_{(1)}\otimes F_{(2)} + 1\otimes F_{(1,2)}$$and $S(F_{(1,2)})=(-1)^3 F_{(2,1)}$.
\end{example}

However, this is not the only basis of $\Qsym$ that will be useful to us. For the second basis we will need to define skew composition diagrams and then standard skew composition tableaux.

For the first of these, let $\alpha, \beta$ be two compositions such that $\beta \lessc \alpha$. Then we define the \emph{skew composition diagram} $\alpha \cskew \beta $ to be the array of boxes that are contained in $\alpha$ but not in $\beta$. That is, the boxes that arise in the saturated chain $\beta \coverc \cdots \coverc \alpha$. We say the \emph{size} of $\alpha \cskew \beta$ is $|\alpha \cskew \beta | = |\alpha| - |\beta|$. Note that if $\beta=\emptyset$, then we recover the composition diagram $\alpha$.

\begin{example}\label{ex:skewshape} The skew composition diagram $(2,1,3)\cskew (1)$ is drawn below with $\beta$ denoted by $\bullet$.
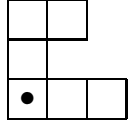
$$\tableau{\ &\ \\
\ \\
\bullet&\ &\ \\}$$
\end{example}

We can now define standard skew composition tableaux. Given a saturated chain, $C$, in $\Lc$
$$\beta = \alpha ^0 \coverc \alpha ^1 \coverc \cdots \coverc \alpha ^{|\alpha \cskew \beta|} = \alpha$$we define the \emph{standard skew composition tableau} $\rtau _C$ of \emph{shape} $\alpha\cskew \beta$ to be the skew composition diagram $\alpha \cskew \beta$ whose boxes are filled with integers such that the number $|\alpha \cskew \beta| -i +1$ appears in the box in $\rtau _C$ that exists in $\alpha ^i$ but not $\alpha ^{i-1}$ for $1\leq i\leq |\alpha \cskew \beta|$. If $\beta = \emptyset$, then we say that we have a \emph{standard composition tableau}. Given a standard skew composition tableau, $\rtau$, whose shape has size $N$ we say that the \emph{descent set} of $\rtau$ is
$$\des (\rtau) = \{ i \suchthat i+1 \mbox{ appears weakly right of } i  \} \subseteq [N-1]$$and the corresponding \emph{descent composition} of $\rtau$ is $\comp(\rtau)= \comp(\des(\rtau))$.

\begin{example}\label{ex:skewCT} The saturated chain 
$$(1)\coverc (2) \coverc (1,2) \coverc (1,1,2) \coverc (1,1,3) \coverc (2,1,3)$$gives rise to the standard skew composition tableau $\rtau$ of shape $(2,1,3)\cskew (1)$ below.
$$\tableau{3 &1\\
4 \\
\bullet&5 &2 \\}$$Note that $\des(\rtau) = \{1,3, 4\}$ and hence $\comp(\rtau) = (1,2,1,1)$.
\end{example}

With this is mind we can now define skew quasisymmetric Schur functions \cite[Proposition 3.1]{BLvW}.

\begin{definition}\label{def:QSbasis}
Let $\alpha \cskew \beta$ be a skew composition diagram. Then the \emph{skew quasisymmetric Schur function} $\qs _{\alpha\cskew \beta}$ is defined to be
$$\qs _{\alpha \cskew \beta} = \sum F_{\comp (\rtau)}$$where the sum is over all standard skew composition tableaux $\rtau$ of shape $\alpha\cskew \beta$. When $\beta = \emptyset$ we call $\qs _\alpha$ a \emph{quasisymmetric Schur function}.
\end{definition}

\begin{example}\label{ex:QSbasis} We can see that $\qs _{(n)} = F_{(n)}$ and $\qs _{(1^n)} = F_{(1^n)}$ and
$$\qs _{(2,1,3)\cskew (1)} = F_{(2,1,2)}+ F_{(2,2,1)} + F_{(1,2,1,1)}$$from the standard skew composition tableaux below.
$$\tableau{2 &1\\
3 \\
\bullet&5 &4 \\}\qquad 
\tableau{2 &1\\
4 \\
\bullet&5 &3 \\}\qquad 
\tableau{3 &1\\
4 \\
\bullet&5 &2 \\}$$
\end{example}

Moreover, the set of all quasisymmetric Schur functions forms another basis for $\Qsym$ such that
$$\Qsym ^N = \spam \{ \qs _\alpha \suchthat |\alpha| = N\}$$and while explicit formulas for their product and antipode are still unknown, their coproduct \cite[Definition 2.19]{BLvW} is given by 
\begin{eqnarray}\label{eq:coproductquasischur}
\Delta(\qs_{\alpha})=\displaystyle\sum_{\gamma} \qs_{\alpha\cskew\gamma}\otimes \qs_{\gamma}
\end{eqnarray}where the sum is over all compositions $\gamma$.  As discussed in the introduction, quasisymmetric Schur functions have many interesting algebraic and combinatorial properties, one of the first of which to be discovered was the exhibition of Pieri rules that utilise our box removing operators \cite[Theorem 6.3]{QS}.

\begin{theorem}\emph{(Pieri rules for quasisymmetric Schur functions)}\label{the:QSPieri}  
Let $\alpha $ be a composition and $n$ be a positive integer. Then
\begin{align*}
\qs_{\alpha}\cdot \qs_{(n)}=\sum  \qs_{\alphap}
\end{align*}
where $\alphap$  is a composition such that $\alpha$ can be obtained by removing an $n$-horizontal strip from it.

Similarly,
\begin{align*}
\qs_{\alpha}\cdot \qs_{(1^n)}=\sum  \qs_{\alphap }
\end{align*}
where $\alphap$  is a composition such that $\alpha$ can be obtained by removing an $n$-vertical strip from it.
\end{theorem}

Dual to $\Qsym$ is the graded Hopf algebra of noncommutative symmetric functions, $\Nsym$, itself a subalgebra of $\bC << x_1, x_2, \ldots >>$ with many interesting bases \cite{GKLLRT}. The one of particular interest to us is the following \cite[Section 2]{BLvW}.

\begin{definition}\label{def:NCbasis} Let $\alpha$ be a composition. Then the \emph{noncommutative Schur function} $\ncs _\alpha$ is the function under the duality pairing $\langle \ ,\ \rangle :\Qsym \otimes \Nsym \rightarrow \bC$ that satisfies 
$$\langle \qs _\alpha , \ncs _\beta \rangle = \delta _{\alpha\beta}$$where $\delta _{\alpha\beta} = 1$ if $\alpha = \beta$ and $0$ otherwise.\end{definition}

Noncommutative Schur functions also have rich and varied algebraic and combinatorial properties, including Pieri rules, although due to the noncommutative nature of $\Nsym$ there are now Pieri rules arising both from multiplication on the right \cite[Theorem 9.3]{tewari}, and from multiplication on the left \cite[Corollary 3.8]{BLvW}. We include them both here for completeness, and for use later.

\begin{theorem}\emph{(Right Pieri rules for noncommutative Schur functions)}\label{the:RightPieri} 
Let $\alpha $ be a composition and $n$ be a positive integer. Then
\begin{align*}
\ncs_{\alpha }\cdot \ncs_{(n)}=\sum \ncs_{\alphap}
\end{align*}
where $\alphap$  is a composition such that it can be obtained by adding an $n$-right horizontal strip to $\alpha$.

Similarly,
\begin{align*}
\ncs_{\alpha }\cdot \ncs_{(1^n)}=\sum \ncs_{\alphap}
\end{align*}
where $\alphap$  is a composition such that it can be obtained by adding an $n$-right vertical strip to $\alpha$.
\end{theorem}

\begin{theorem}\emph{(Left Pieri rules for noncommutative Schur functions)}\label{the:LeftPieri} 
Let $\alpha $ be a composition and $n$ be a positive integer. Then
\begin{align*}
\ncs_{(n)} \cdot \ncs_{\alpha } =\sum \ncs_{\alphap}
\end{align*}
where $\alphap$  is a composition such that it can be obtained by adding an $n$-left horizontal strip to $\alpha$.

Similarly,
\begin{align*}
\ncs_{(1^n)} \cdot \ncs_{\alpha }  =\sum \ncs_{\alphap }
\end{align*}
where $\alphap$  is a composition such that it can be obtained by adding an $n$-left vertical strip to $\alpha$.\end{theorem}

Note that since quasisymmetric and noncommutative Schur functions are indexed by  compositions, \emph{if any parts of size 0 arise during computation, then they are ignored}.

\section{Generalized skew Pieri rules}\label{sec:skew} 
\subsection{Quasisymmetric skew Pieri rules}\label{subsec:QSymskewPieri} We now turn our attention to proving skew Pieri rules for skew quasisymmetric Schur functions. The statement of the rules is in the spirit of the Pieri rules for skew shapes of Assaf and McNamara \cite{assaf-mcnamara}, and this is no coincidence as we recover their rules as a special case in Corollary~\ref{cor:AM}. However first we prove a crucial proposition.

\begin{proposition}\label{ob:skewingisharpooning} Let $\alpha, \beta$ be compositions. Then
$\ncsb \rightharpoonup \qs_{\alpha} = \qs_{\alpha\cskew\beta}$.
\end{proposition}

\begin{proof}
Recall  Equation~\eqref{eq:coproductquasischur} states that
$$
\Delta(\qs_{\alpha})=\displaystyle\sum_{\gamma} \qs_{\alpha\cskew\gamma}\otimes \qs_{\gamma}
$$
where the sum is over all compositions $\gamma$. Thus using Equations \eqref{eq:HdualactingonH} and \eqref{eq:coproductquasischur} we obtain
\begin{eqnarray}\label{eq:harpoons}
\ncsb \rightharpoonup \qs_{\alpha}= \displaystyle\sum_{\gamma} \langle \qs_{\gamma},\ncsb\rangle \qs_{\alpha\cskew\gamma}
\end{eqnarray}
where the sum is over all compositions $\gamma$. Since by Definition~\ref{def:NCbasis}, $\langle \qs_{\gamma},\ncsb\rangle$ equals $1$ if $\beta= \gamma$ and  $0$ otherwise, the claim follows.
\end{proof}

\begin{remark}\label{rem:terms that equal 0 for poset reasons}
The proposition above does not tell us when $\ncsb \rightharpoonup \qs_{\alpha} = \qs_{\alpha\cskew\beta}$ equals $0$. However, by the definition of $\alpha \cskew \beta$ this is precisely when $\alpha$ and $\beta$ satisfy $\beta \not\lessc \alpha$. Consequently in the theorem below the \emph{nonzero} contribution will only be from those $\alphap$ and $\betam$ that satisfy $\betam \lessc \alphap$. As always if any parts of size 0 arise during computation, then they are ignored.
\end{remark}

\begin{theorem}\label{the:QSskewPieri}
Let $\alpha, \beta$ be compositions and $n$ be a positive integer. Then
\begin{align*}
\qs_{\alpha\cskew\beta}\cdot \qs_{(n)}=\sum_{i+j=n}(-1)^j\qs_{\alphap\cskew\betam}
\end{align*}
where $\alphap$  is a composition such that $\alpha$ can be obtained by removing an $i$-horizontal strip from it, and $\beta^{-}$ is a composition such that it can be obtained by removing a $j$-vertical strip from $\beta$.

Similarly,
\begin{align*}
\qs_{\alpha\cskew\beta}\cdot \qs_{(1^n)}=\sum_{i+j=n}(-1)^j\qs_{\alphap\cskew\betam}
\end{align*}
where $\alphap$  is a composition such that $\alpha$ can be obtained by removing an $i$-vertical strip from it, and $\beta^{-}$ is a composition such that it can be obtained by removing a $j$-horizontal strip from $\beta$.
\end{theorem}

\begin{proof}
For the first part of the theorem, our aim is to calculate $\qs_{\alpha\cskew\beta}\cdot \qs_{(n)}$, which in light of Proposition \ref{ob:skewingisharpooning}, is the same as calculating $(\ncsb \rightharpoonup \qs_{\alpha})\cdot \qs_{(n)}$. 

Taking $a=\ncsb$, $g=\qs_{\alpha}$ and $h=\qs_{(n)}$ in Lemma \ref{lem:magiclemma} gives the LHS as $(\ncsb \rightharpoonup \qs_{\alpha})\cdot \qs_{(n)}$. For the RHS observe that, by Definition~\ref{def:QSbasis}, $\qs _{(n)} = {F} _{(n)}$ and by Equation~\eqref{eq:Fcoproduct} we have that
\begin{eqnarray}\label{eq:coproductF}
\Delta({F}_{(n)})&=& \sum_{i+j=n}{F}_{(i)}\otimes {F}_{(j)}.
\end{eqnarray}

Substituting these in yields 
\begin{eqnarray}\label{eq:firststeprhs}
\displaystyle\sum_{i+j=n}(S({F}_{(j)})\rightharpoonup \ncsb)\rightharpoonup(\qs_{\alpha}\cdot {F}_{(i)}).
\end{eqnarray}

Now, by Equation~\eqref{eq:antipode}, we have that $S({F}_{(j)})=(-1)^j{F}_{(1^j)}$.
This reduces \eqref{eq:firststeprhs} to 
\begin{eqnarray}\label{eq:secondsteprhs}
\displaystyle\sum_{i+j=n}((-1)^j{F}_{(1^j)}\rightharpoonup \ncsb)\rightharpoonup(\qs_{\alpha}\cdot {F}_{(i)}).
\end{eqnarray}

We will first deal with the task of evaluating ${F}_{(1^j)}\rightharpoonup \ncsb$. We need to invoke Equation \eqref{eq:HactingonHdual} and thus we need $\Delta(\ncsb)$. Assume that 
\begin{eqnarray}
\Delta(\ncsb)=\sum_{\gamma,\delta}b_{\gamma,\delta}^{\beta}\ncs_{\gamma}\otimes \ncs_{\delta}
\end{eqnarray}
where the sum is over all compositions $\gamma,\delta$.
Thus Equation~\eqref{eq:HactingonHdual} yields
\begin{eqnarray}
{F}_{(1^j)}\rightharpoonup \ncsb&=& \displaystyle\sum_{\gamma,\delta} b_{\gamma,\delta}^{\beta}\langle {F}_{(1^j)},\ncs_{\delta}\rangle\ncs_{\gamma}.
\end{eqnarray}
Observing that, by Definition~\ref{def:QSbasis}, ${F}_{(1^j)}=\qs_{(1^j)}$ and that, by Definition~\ref{def:NCbasis}, $\langle \qs_{(1^j)},\ncs_{\delta}\rangle$ equals $1$ if $\delta = (1^j)$ and equals  $0$ otherwise, we obtain
\begin{eqnarray}\label{eq:eharpooningncsstep1}
{F}_{(1^j)}\rightharpoonup \ncsb &=& \displaystyle\sum_{\gamma} b_{\gamma,(1^j)}^{\beta}\ncs_{\gamma}.
\end{eqnarray}
Since by Definition~\ref{def:NCbasis} and the duality pairing we have that $\langle  \qs_{\gamma}\otimes \qs_{\delta},\Delta(\ncsb)\rangle=\langle \qs_{\gamma}\cdot\qs_{\delta},\ncsb\rangle=b_{\gamma,\delta}^{\beta}$, we get that
\begin{eqnarray}
\langle \qs_{\gamma}\cdot\qs_{(1^j)},\ncsb\rangle&=& b_{\gamma,(1^j)}^{\beta}.
\end{eqnarray}

The Pieri rules for quasisymmetric Schur functions in Theorem~\ref{the:QSPieri} state that $b_{\gamma,(1^j)}^{\beta}$ is $1$ if there exists a weakly decreasing sequence $\ell_1\geq \ell_2\geq\cdots \geq \ell_j$ such that $\down_{\ell_1}\cdots \down_{\ell_j}(\beta)=\gamma$, and is $0$ otherwise. Thus this reduces Equation~\eqref{eq:eharpooningncsstep1} to
\begin{eqnarray}\label{eq:eharpooningncsstep2}
{F}_{(1^j)}\rightharpoonup \ncsb &=& \displaystyle\sum_{\substack{\down_{\ell_1}\cdots \down_{\ell_j}(\beta)=\gamma\\\ell_1\geq\cdots \geq \ell_j}}\ncs_{\gamma}.
\end{eqnarray}

Since $\qs_{(i)}={F}_{(i)}$, by Definition~\ref{def:QSbasis}, the Pieri rules in Theorem~\ref{the:QSPieri} also imply that
\begin{eqnarray}\label{eq:rowpieriruleqschur}
\qs_{\alpha}\cdot {F}_{(i)}&=& \displaystyle \sum_{\substack{\down_{r_1}\cdots \down_{r_i}(\varepsilon)=\alpha\\r_1<\cdots < r_i}}\qs_{\varepsilon}.
\end{eqnarray}

Using Equations~\eqref{eq:eharpooningncsstep2} and \eqref{eq:rowpieriruleqschur} in \eqref{eq:secondsteprhs}, we get 
\begin{eqnarray}
\displaystyle\sum_{i+j=n}((-1)^j{F}_{(1^j)}\rightharpoonup \ncsb)\rightharpoonup(\qs_{\alpha}\cdot {F}_{(i)})&=& \displaystyle\sum_{i+j=n}\left((-1)^j\displaystyle\sum_{\substack{\down_{\ell_1}\cdots \down_{\ell_j}(\beta)=\gamma\\\ell_1\geq\cdots \geq \ell_j}}\ncs_{\gamma}\right)\rightharpoonup \left( \sum_{\substack{\down_{r_1}\cdots \down_{r_i}(\varepsilon)=\alpha\\r_1<\cdots < r_i}}\qs_{\varepsilon}
\right).
\nonumber\\
\end{eqnarray}
Using Proposition \ref{ob:skewingisharpooning}, we obtain that
\begin{eqnarray}\label{eq:penultimateskewpieri}
\displaystyle\sum_{i+j=n}((-1)^j{F}_{(1^j)}\rightharpoonup \ncsb)\rightharpoonup(\qs_{\alpha}\cdot {F}_{(i)})&=&\displaystyle\sum_{i+j=n}\left(\sum_{\substack{\down_{\ell_1}\cdots \down_{\ell_j}(\beta)=\gamma\\\ell_1\geq\cdots \geq \ell_j\\\down_{r_1}\cdots \down_{r_i}(\varepsilon)=\alpha\\r_1<\cdots < r_i}}(-1)^j\qs_{\varepsilon\cskew\gamma}\right).
\end{eqnarray}

Thus
\begin{eqnarray}\label{eq:skewpierirule-row}
\qs_{\alpha\cskew\beta}\cdot \qs_{(n)}&=&\displaystyle\sum_{i+j=n}\left(\sum_{\substack{\down_{\ell_1}\cdots \down_{\ell_j}(\beta)=\gamma\\\ell_1\geq\cdots \geq \ell_j\\\down_{r_1}\cdots \down_{r_i}(\varepsilon)=\alpha\\r_1<\cdots < r_i}}(-1)^j\qs_{\varepsilon\cskew\gamma}\right).
\end{eqnarray} The first part of the theorem now follows from the definitions of $i$-horizontal strip and $j$-vertical strip.

For the second part of the theorem we use the same method as the first part, but this time calculate
$$(\ncsb \rightharpoonup \qs_{\alpha})\cdot \qs_{(1^n)}.$$
\end{proof}

\begin{remark}\label{rem:noomega}
Notice that as opposed to the classical case where one can apply the $\omega$ involution to obtain the corresponding Pieri rule, we can not do this here. This is because the image of the skew quasisymmetric Schur functions under the $\omega$ involution is not yet known explicitly. Notice that the $\omega$ map applied to quasisymmetric Schur functions results in the row-strict quasisymmetric Schur functions of Mason and Remmel \cite{mason-remmel}.
\end{remark}


\begin{example}\label{ex:skewQSPierirow}
Let us compute $\qs_{(1,3,2)\cskew (2,1)}\cdot \qs_{(2)}$. We first need to compute all compositions $\gamma$ that can be obtained by removing a vertical strip of size at most 2 from $\beta=(2,1)$. These compositions correspond to the white boxes in the diagrams below, while the boxes in the darker shade of red correspond to the vertical strips that are removed from $\beta$.

$$
\ytableausetup{smalltableaux,boxsize=0.5em}
\begin{ytableau}
*(white) & *(white)\\
*(white)
\end{ytableau}
\hspace{5mm}
\begin{ytableau}
*(white) & *(white)\\
*(red!80)
\end{ytableau}
\hspace{5mm}
\begin{ytableau}
*(white) & *(red!80)\\
*(white)
\end{ytableau}
\hspace{5mm}
\begin{ytableau}
*(white) & *(red!80)\\
*(red!80)
\end{ytableau}
$$

Next we need to compute all compositions $\varepsilon$ such that a horizontal strip of size at most $2$ can be removed from it so as to obtain $\alpha$. We list these $\varepsilon$s below with the boxes in the lighter shade of green corresponding to horizontal strips that need to be removed to obtain $\alpha$.

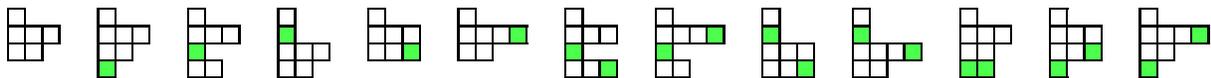
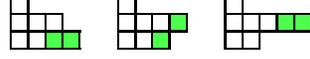
$$
\ytableausetup{smalltableaux,boxsize=0.5em}
\begin{ytableau}
*(white)\\
*(white) & *(white) & *(white)\\
*(white) & *(white)
\end{ytableau}
\hspace{5mm}
\begin{ytableau}
*(white)\\
*(white) & *(white) & *(white)\\
*(white) & *(white)\\
*(green!70)\\
\end{ytableau}
\hspace{5mm}
\begin{ytableau}
*(white)\\
*(white) & *(white) & *(white)\\
*(green!70)\\
*(white) & *(white)
\end{ytableau}
\hspace{5mm}
\begin{ytableau}
*(white)\\
*(green!70)\\
*(white) & *(white) & *(white)\\
*(white) & *(white)
\end{ytableau}
\hspace{5mm}
\begin{ytableau}
*(white)\\
*(white) & *(white) & *(white)\\
*(white) & *(white) & *(green!70)\\
\end{ytableau}
\hspace{5mm}
\begin{ytableau}
*(white)\\
*(white) & *(white) & *(white) &*(green!70)\\
*(white) & *(white)
\end{ytableau}
\hspace{5mm}
\begin{ytableau}
*(white)\\
*(white) & *(white) & *(white) \\
*(green!70)\\
*(white) & *(white) & *(green!70)
\end{ytableau}
\hspace{5mm}
\begin{ytableau}
*(white)\\
*(white) & *(white) & *(white) & *(green!70)\\
*(green!70)\\
*(white) & *(white) 
\end{ytableau}
\hspace{5mm}
\begin{ytableau}
*(white)\\
*(green!70)\\
*(white) & *(white) & *(white)\\
*(white) & *(white) & *(green!70)\\
\end{ytableau}
\hspace{5mm}
\begin{ytableau}
*(white)\\
*(green!70)\\
*(white) & *(white) & *(white) & *(green!70)\\
*(white) & *(white) 
\end{ytableau}
\hspace{5mm}
\begin{ytableau}
*(white)\\
*(white) & *(white) & *(white) \\
*(white) & *(white) \\
*(green!70) & *(green!70)\\
\end{ytableau}
\hspace{5mm}
\begin{ytableau}
*(white)\\
*(white) & *(white) & *(white) \\
*(white) & *(white) & *(green!70)\\
*(green!70) \\
\end{ytableau}
\hspace{5mm}
\begin{ytableau}
*(white)\\
*(white) & *(white) & *(white) & *(green!70) \\
*(white) & *(white) \\
*(green!70) \\
\end{ytableau}
\hspace{5mm}
$$
$$
\ytableausetup{smalltableaux,boxsize=0.5em}
\begin{ytableau}
*(white)\\
*(white) & *(white) & *(white)  \\
*(white) & *(white)  & *(green!70) & *(green!70)\\
\end{ytableau}
\hspace{5mm}
\begin{ytableau}
*(white)\\
*(white) & *(white) & *(white) & *(green!70) \\
*(white) & *(white)  & *(green!70)\\
\end{ytableau}
\hspace{5mm}
\begin{ytableau}
*(white)\\
*(white) & *(white) & *(white)  & *(green!70) & *(green!70)\\
*(white) & *(white)  \\
\end{ytableau}
$$

Now to compute $\qs_{(1,3,2)\cskew (2,1)}\cdot \qs_{(2)}$, our result tells us that for every pair of compositions in the above lists $(\varepsilon,\gamma)$ such that (the number of green boxes in $\varepsilon$)+(the number of red boxes in $\gamma$)=2, and $\gamma \lessc \varepsilon$ we have a term $\qs_{\varepsilon \cskew \gamma}$ with a sign $(-1)^{\text{number of red boxes}}$.  Hence  we have the following expansion, suppressing commas and parentheses in  compositions for ease of comprehension.
\begin{align*}
\qs_{132\cskew 21}\cdot \qs_{2}=&\qs_{132\cskew 1}-\qs_{1321\cskew 11}-\qs_{1312\cskew 2}-\qs_{1132\cskew2}-\qs_{1132\cskew 11}-\qs_{133\cskew 2}-\qs_{133\cskew 11}\\&-\qs_{142\cskew 2}-\qs_{142\cskew 11}+\qs_{1133\cskew 21}+\qs_{1142\cskew 21}+\qs_{1322\cskew 21}+\qs_{1331\cskew 21}\\&+\qs_{1421\cskew 21}+\qs_{143\cskew 21}+\qs_{152\cskew 21}
\end{align*}
\end{example}


\begin{example}\label{ex:skewQSPiericol}
Let us compute $\qs_{(1,3,2)\cskew (2,1)}\cdot \qs_{(1,1)}$. We first need to compute all compositions $\gamma$ that can be obtained by removing a horizontal strip of size at most 2 from $\beta=(2,1)$. These compositions correspond to the white boxes in the diagrams below, while the boxes in the darker shade of red correspond to the horizontal strips that are removed from $\beta$.

$$
\ytableausetup{smalltableaux,boxsize=0.5em}
\begin{ytableau}
*(white) & *(white)\\
*(white)
\end{ytableau}
\hspace{5mm}
\begin{ytableau}
*(white) & *(white)\\
*(red!80)
\end{ytableau}
\hspace{5mm}
\begin{ytableau}
*(white) & *(red!80)\\
*(white)
\end{ytableau}
\hspace{5mm}
\begin{ytableau}
*(white) & *(red!80)\\
*(red!80)
\end{ytableau}
$$

Next we need to compute all compositions $\varepsilon$ such that a vertical strip of size at most $2$ can be removed from it so as to obtain $\alpha$. We list these $\varepsilon$s below with the boxes in the lighter shade of green corresponding to vertical strips that need to be removed to obtain $\alpha$.

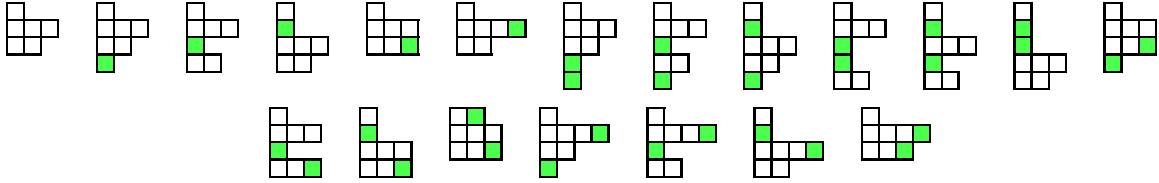
$$
\ytableausetup{smalltableaux,boxsize=0.5em}
\begin{ytableau}
*(white)\\
*(white) & *(white) & *(white)\\
*(white) & *(white)
\end{ytableau}
\hspace{5mm}
\begin{ytableau}
*(white)\\
*(white) & *(white) & *(white)\\
*(white) & *(white)\\
*(green!70)\\
\end{ytableau}
\hspace{5mm}
\begin{ytableau}
*(white)\\
*(white) & *(white) & *(white)\\
*(green!70)\\
*(white) & *(white)
\end{ytableau}
\hspace{5mm}
\begin{ytableau}
*(white)\\
*(green!70)\\
*(white) & *(white) & *(white)\\
*(white) & *(white)
\end{ytableau}
\hspace{5mm}
\begin{ytableau}
*(white)\\
*(white) & *(white) & *(white)\\
*(white) & *(white) & *(green!70)\\
\end{ytableau}
\hspace{5mm}
\begin{ytableau}
*(white)\\
*(white) & *(white) & *(white) &*(green!70)\\
*(white) & *(white)
\end{ytableau}
\hspace{5mm}
\begin{ytableau}
*(white)\\
*(white) & *(white) & *(white)\\
*(white) & *(white) \\
*(green!70)\\
*(green!70)\\
\end{ytableau}
\hspace{5mm}
\begin{ytableau}
*(white)\\
*(white) & *(white) & *(white)\\
*(green!70)\\
*(white) & *(white) \\
*(green!70)\\
\end{ytableau}
\hspace{5mm}
\begin{ytableau}
*(white)\\
*(green!70)\\
*(white) & *(white) & *(white)\\
*(white) & *(white) \\
*(green!70)\\
\end{ytableau}
\hspace{5mm}
\begin{ytableau}
*(white)\\
*(white) & *(white) & *(white)\\
*(green!70)\\
*(green!70)\\
*(white) & *(white) \\
\end{ytableau}
\hspace{5mm}
\begin{ytableau}
*(white)\\
*(green!70)\\
*(white) & *(white) & *(white)\\
*(green!70)\\
*(white) & *(white) \\
\end{ytableau}
\hspace{5mm}
\begin{ytableau}
*(white)\\
*(green!70)\\
*(green!70)\\
*(white) & *(white) & *(white)\\
*(white) & *(white) \\
\end{ytableau}
\hspace{5mm}
\begin{ytableau}
*(white)\\
*(white) & *(white) & *(white) \\
*(white) & *(white) & *(green!70) \\
*(green!70) \\
\end{ytableau}
\hspace{5mm}
$$
$$
\ytableausetup{smalltableaux,boxsize=0.5em}
\begin{ytableau}
*(white)\\
*(white) & *(white) & *(white)  \\
*(green!70)\\
*(white) & *(white)  & *(green!70)\\
\end{ytableau}
\hspace{5mm}
\begin{ytableau}
*(white)\\
*(green!70)\\
*(white) & *(white) & *(white)  \\
*(white) & *(white)  & *(green!70) \\
\end{ytableau}
\hspace{5mm}
\begin{ytableau}
*(white) & *(green!70)\\
*(white) & *(white) & *(white)   \\ 
*(white) & *(white)  & *(green!70)\\
\end{ytableau}
\hspace{5mm}
\begin{ytableau}
*(white) \\
*(white) & *(white) & *(white)  & *(green!70)\\  
*(white) & *(white) \\
*(green!70)\\
\end{ytableau}
\hspace{5mm}
\begin{ytableau}
*(white) \\
*(white) & *(white) & *(white)  & *(green!70)\\  
*(green!70)\\
*(white) & *(white) \\
\end{ytableau}
\hspace{5mm}
\begin{ytableau}
*(white) \\
*(green!70)\\
*(white) & *(white) & *(white)   & *(green!70)\\ 
*(white) & *(white)  \\
\end{ytableau}
\hspace{5mm}
\begin{ytableau}
*(white) \\
*(white) & *(white) & *(white) & *(green!70)  \\ 
*(white) & *(white)  & *(green!70)\\
\end{ytableau}
$$

Now to compute $\qs_{(1,3,2)\cskew (2,1)}\cdot \qs_{(1,1)}$, our result tells us that for every pair of compositions in the above lists $(\varepsilon,\gamma)$ such (that the number of green boxes in $\varepsilon$)+(the number of red boxes in $\gamma$)=2 and $\gamma \lessc \varepsilon$, we have a term $\qs_{\varepsilon\cskew \gamma}$ with a sign $(-1)^{\text{number of red boxes}}$. Hence  we have the following expansion, suppressing commas and parentheses in  compositions for ease of comprehension.
\begin{align*}
\qs_{132\cskew 21}\cdot \qs_{11}=&\qs_{132\cskew 1}-\qs_{1321\cskew 11}-\qs_{1312\cskew 2}-\qs_{1132\cskew2}-\qs_{1132\cskew 11}-\qs_{133\cskew 2}-\qs_{133\cskew 11}\\&-\qs_{142\cskew 2}-\qs_{142\cskew 11}+\qs_{13121\cskew 21}+\qs_{11321\cskew 21}+\qs_{11132\cskew 21}+\qs_{1331\cskew 21}\\&+\qs_{1133\cskew 21}+\qs_{233\cskew 21}+\qs_{1421\cskew 21}+\qs_{1142\cskew 21}+\qs_{143\cskew 21}
\end{align*}
\end{example}

We now turn our attention to skew Schur functions, which we will define in the next paragraph, after we first discuss some motivation for our attention.
Skew Schur functions can be written as a sum of skew quasisymmetric Schur functions \cite[Lemma 2.23]{SSQSS}, so one might ask whether we can recover the Pieri rules for skew shapes of Assaf and McNamara by expanding a skew Schur function as a sum of skew quasisymmetric Schur functions, applying our quasisymmetric skew Pieri rules and then collecting suitable terms. However, a much simpler proof exists.


A skew Schur function $s _{\lambda/\mu}$ for partitions $\lambda, \mu$ where $\ell(\lambda)\geq\ell(\mu)$, can be defined, given an $M>\ell(\lambda)$, by \cite[Section 5.1]{BLvW}
\begin{equation}\label{eq:skewSchur}s _{\lambda/\mu}=\qs _{\lambda + 1^{M} \cskew \mu + 1^{M}}\end{equation}where  $\lambda + 1^{M} = (\lambda _1 + 1, \ldots, \lambda_{\ell(\lambda)} +1, 1^{M-\ell(\lambda)})$, and $\mu + 1^{M} = (\mu _1 + 1, \ldots, \mu_{\ell(\mu)} +1, 1^{M -\ell(\mu)})$. It follows immediately that $s_{(n)}=s_{(n)/\emptyset}=\qs _{(n)}$ and $s_{(1^n)}=s_{(1^n)/\emptyset}=\qs _{(1^n)}$ by Equation~\eqref{eq:skewSchur}. Then as a corollary of our skew Pieri rules we recover the skew Pieri rules of Assaf and McNamara as follows.

\begin{corollary}\cite[Theorem 3.2]{assaf-mcnamara}\label{cor:AM} Let $\lambda, \mu$ be partitions where $\ell(\lambda)\geq\ell(\mu)$ and $n$ be a positive integer. Then
$$s_{\lambda/\mu}\cdot s_{(n)} = \sum _{i+j = n} (-1) ^j s _{\lambda^+/\mu ^-}$$ where $\lambda^+$ is a partition such that the boxes of $\lambda ^+$ not in $\lambda$ are $i$ boxes such that no two lie in the same column, and $\mu ^-$ is a partition such that the boxes of $\mu$ not in $\mu ^-$ are $j$ boxes such that no two lie in the same row. 

Similarly,
$$s_{\lambda/\mu}\cdot s_{(1^n)} = \sum _{i+j = n} (-1) ^j s _{\lambda^+/\mu ^-}$$ where $\lambda^+$ is a partition such that the boxes of $\lambda ^+$ not in $\lambda$ are $i$ boxes such that no two lie in the same row, and $\mu ^-$ is a partition such that the boxes of $\mu$ not in $\mu ^-$ are $j$ boxes such that no two lie in the same column. 
\end{corollary}

\begin{proof} Let $N>\ell(\lambda)+n+1$. Then consider the product $\qs _{\lambda + 1^N \cskew \mu + 1^N} \cdot  \qs _{(n)}$ (respectively, $\qs _{\lambda + 1^N \cskew \mu + 1^N} \cdot  \qs _{(1^n)}$) where $\lambda, \mu$ are partitions and $\ell(\lambda)\geq\ell(\mu)$. By the paragraph preceding the corollary, this is equivalent to what we are trying to compute.

To begin, we claim  that if 
$$\rem _1\rem _{r_2} \cdots \rem _{r_i} (\alpha ')= \lambda + 1^N$$where $1<r_2<\cdots<r_i$ (respectively, $\rem _{r _1} \cdots \rem _{r _{i-q}} \rem _1^q (\alpha ')= \lambda + 1^N$ where $q\geq 0$ and $r _1 \geq \cdots \geq r _{i-q} >1$) and 
$$\rem _{\ell _1} \cdots \rem _{\ell _{j-p}} \rem _1^p (\mu+ 1^N) = \beta '$$ where $p\geq 0$ and $\ell _1 \geq \cdots \geq \ell _{j-p} >1$ (respectively, $\rem _1\rem _{\ell_2} \cdots \rem _{\ell_j} (\mu+ 1^N) = \beta '$ where $1<\ell_2<\cdots< \ell_j$) and
$$\rem _{r_2} \cdots \rem _{r_i} (\alpha '')= \lambda + 1^N$$(respectively, $\rem _{r _1} \cdots \rem _{r _{i-q}} \rem _1^{q+1}(\alpha '')= \lambda + 1^N$) and 
$$\rem _{\ell _1} \cdots \rem _{\ell _{j-p}} \rem _1^{p+1} (\mu+ 1^N) = \beta ''$$(respectively, $\rem _{\ell_2} \cdots \rem _{\ell_j}(\mu+ 1^N) = \beta ''$) then $\beta ' \lessc \alpha ' $ if and only if $\beta '' \lessc \alpha ''$. This follows from three facts. Firstly $\ell(\lambda) \geq \ell(\mu)$, secondly $\alpha ' $ and $\alpha ''$ only differ by $\down _1$, and thirdly $\beta '$ and  $\beta ''$ only differ by $\down _1$. Moreover, ${\alpha ' \cskew \beta '}= {\alpha '' \cskew \beta ''}$. Furthermore, by our skew Pieri rules in Theorem~\ref{the:QSskewPieri}, the summands $\qs _{\alpha ' \cskew \beta '}$ and $\qs_{\alpha '' \cskew \beta ''}$ will be of opposite sign, and thus will cancel since ${\alpha ' \cskew \beta '}= {\alpha '' \cskew \beta ''}$. Consequently, any nonzero summand appearing in the product $\qs _{\lambda + 1^N \cskew \mu + 1^N} \cdot  \qs _{(n)}$ (respectively, $\qs _{\lambda + 1^N \cskew \mu + 1^N} \cdot  \qs _{(1^n)}$) is such that no box can be removed from the first column of $(\lambda + 1^N)^+$ to obtain $\lambda + 1^N$, nor from the first column of $\mu+ 1^N$ to obtain $(\mu+ 1^N)^-$.

Next observe that we can obtain $\lambda + 1^N$ by removing an $i$-horizontal (respectively, $i$-vertical) strip not containing a box in the first column from $(\lambda + 1^N)^+$  if and only if $(\lambda + 1^N)^+ = \lambda ^+ + 1^N$ where $\lambda ^+$ is a partition such that the boxes of $\lambda ^+$ not in $\lambda$  are $i$ boxes such that no two lie in the same column (respectively, row).

Similarly, we can obtain $(\mu+ 1^N)^-$ by removing a $j$-vertical (respectively, $j$-horizontal) strip not containing a box in the first column  from $\mu+ 1^N$ if and only if $(\mu+ 1^N)^- = \mu^-+ 1^N$ where $\mu ^-$ is a partition such that the boxes of $\mu$ not in $\mu ^-$ are $j$ boxes such that no two lie in the same row (respectively, column).
\end{proof}

\subsection{Noncommutative skew Pieri rules}\label{subsec:NSymskewPieri} It is also natural to ask whether skew Pieri rules exist for the dual counterparts to skew quasisymmetric Schur functions and whether our methods are applicable in order to prove them. To answer this we first need to define these dual counterparts, namely skew noncommutative Schur functions.

\begin{definition}\label{def:ncsskew} Given compositions $\alpha, \beta$, the \emph{skew noncommutative Schur function} $\ncs_{\alpha/\beta}$ is defined implicitly  via the equation
\begin{eqnarray*}
\Delta(\ncsa)&=&\displaystyle\sum_{\beta}\ncs_{\alpha/\beta}\otimes \ncsb
\end{eqnarray*}
where the sum ranges over all compositions $\beta$.
\end{definition}

With this definition and using Equation~\eqref{eq:HactingonHdual} we can deduce that
$$\qs_{\beta} \rightharpoonup \ncsa= \ncs_{\alpha/\beta}$$via a proof almost identical to that of Proposition~\ref{ob:skewingisharpooning}. We know from Definition~\ref{def:QSbasis} that  $\qs _{(n)} = F_{(n)}$ and $\qs _{(1^n)} = F_{(1^n)}$. Combined with the product for fundamental quasisymmetric functions using Definition~\ref{def:Fbasis}, Definition~\ref{def:NCbasis}, and the duality pairing, it is straightforward to deduce that for $n\geq 1$ the coproduct on $\ncs _{(n)}$ and $\ncs _{(1^n)}$ is given by 
$$\Delta(\ncs _{(n)}) = \sum _{i+j=n} \ncs _{(i)}\otimes \ncs _{(j)} \qquad\Delta(\ncs _{(1^n)}) = \sum _{i+j=n} \ncs _{(1^i)}\otimes \ncs _{(1^j)}.$$Also the action of the antipode $S$ on $\ncs _{(n)}$ and $\ncs _{(1^n)}$ is given by
$$S(\ncs _{(j)})= (-1)^j \ncs _{(1^j)} \qquad S(\ncs _{(1^j)})= (-1)^j \ncs _{(j)}.$$Using all the above in conjunction with the  right Pieri rules for noncommutative Schur functions in Theorem~\ref{the:RightPieri} yields our concluding theorem, whose proof is analogous to the proof of Theorem~\ref{the:QSskewPieri}, and hence is omitted. As always if any parts of size 0 arise during computation, then they are ignored.

\begin{theorem}\label{the:NCskewPieri}
Let $\alpha, \beta$ be compositions and $n$ be a positive integer. Then
\begin{align*}
\ncs_{\alpha/\beta}\cdot \ncs_{(n)}=\sum_{i+j=n}(-1)^j\ncs_{\alphap/\betam}
\end{align*}
where $\alphap$  is a composition such that it can be obtained by adding an $i$-right horizontal strip to $\alpha$, and $\beta^{-}$ is a composition such that $\beta$ can be obtained by adding a $j$-right vertical strip to it.

Similarly,
\begin{align*}
\ncs_{\alpha/\beta}\cdot \ncs_{(1^n)}=\sum_{i+j=n}(-1)^j\ncs_{\alphap/\betam}
\end{align*}
where $\alphap$  is a composition such that it can be obtained by adding an $i$-right vertical strip to $\alpha$, and $\beta^{-}$ is a composition such that $\beta$ can be obtained by adding a $j$-right horizontal strip to it.
\end{theorem}

\section*{Acknowledgements} The authors would like to thank the referees for helpful suggestions.


\def\cprime{$'$}

\end{document}